\documentclass{article}
\makeatletter
\newcommand{\crossref}[1]{}
\newcommand{\ams}[1]{}
\newcommand{\mi}[1]{}
\newcommand{\zmath}[1]{}
\newcommand{\ads}[1]{}
\makeatother

\usepackage[utf8]{inputenc}
\usepackage[english,russian]{babel}
\usepackage[english]{babel}
\usepackage[T1]{fontenc} 
\usepackage[tbtags]{amsmath}
\usepackage{amsfonts,amssymb,mathrsfs,amscd}
\usepackage{titleps}

\let\No=\textnumero

\usepackage[hyper]{izv2e}

\numberwithin{equation}{section}

\theoremstyle{plain}
\newtheorem{theorem}{Theorem}
\newtheorem{lemma}{Lemma}[section]
\newtheorem{propos}{Proposition}

\newtheorem{coro}{Corollary}

\theoremstyle{definition}

\newtheorem{proof}{Proof}

\newtheorem{remark}{Remark}

\def\acts{\curvearrowright}
\def\em{\varnothing}

\def\Nb{\mathbb N}

\def\Cb{\mathbb C}
\def\Tb{\mathbb T}

\def\A{\mathcal A}
\def\Cc{\mathcal C}
\def\U{\mathcal U}

\def\L{\mathcal L}
\def\R{\mathcal R}

\def\W{\mathcal W}

\def\sm{\setminus}
\def\ovl{\overline}
\def\phi{\varphi}
\def\epsilon{\varepsilon}

\DeclareMathOperator*{\Nabla}{\nabla}

\newpagestyle{myheaders}{
	\sethead{}
	{\ifodd\value{page}G.B. Sorin\else Roelcke and WAP compactifications of automorphism groups\fi}
	{\thepage}
}

\newcommand{\myname}{G.B. Sorin}

\begin{document}
	\pagestyle{myheaders}
	\thispagestyle{empty}
	
	\title{Roelcke and WAP compactifications of automorphism groups of ultrahomogeneous cyclically ordered sets}
	\author{\texorpdfstring{G.\,B.~Sorin}{G. B. Sorin}}
	\address{Lomonosov Moscow State University}
	\email{georgsorin@yandex.ru}

	\maketitle
	
	\begin{fulltext}
		
		\begin{abstract}
			In this work we describe Roelcke and WAP compactifications of automorphism groups of discrete ultrahomogeneous cyclically ordered sets in the topology of pointwise convergence.
			
			
		\end{abstract}

\section{Introduction}\label{section-intro}
	
	Among the compactifications of a topological group $G$, there are those that arise naturally. Examples of such compactifications are the Roelcke compactification $b_rG$ and the WAP compactification $wG$. These can be described as the Gelfand spaces of the corresponding $C^*$-algebras. In the case of $b_rG$, this is the algebra $UC(G)$ of all uniformly continuous functions $(G,\L\land \R)\to \Cb$, where $\L\land \R$ denotes the Roelcke uniformity on the group $G$. The compactification $wG$ corresponds to the algebra $WAP(G)$ of weakly almost periodic functions on G.	The compactification $wG$ is the largest in the class of semitopological semigroup compactifications of $G$. While the Roelcke compactification is always proper (i.e., $G$ is homeomorphically embedded as a dense subset in $b_rG$), the WAP compactification may be improper (i.e., $G$ is continuously mapped onto a dense subset of $wG$). 	The first example of a group with a trivial (consisting of a single point) WAP compactification is the automorphism group of a line segment $Aut(I,<)$ in the topology of pointwise convergence \cite{MegWapTriv}. From the inclusion $UC(G)\supset WAP(G)$, the inequality of compactifications $b_rG \geq wG$ follows. Some examples and properties of $b_rG$ and $wG$ can be found in \cite{Usp0}.
	
	An explicit description of $b_rG$ in several cases allows one to establish the topological simplicity or minimality of G (see \cite{Usp1, Usp2, Usp3}). Functions from $WAP(G)$ are matrix coefficients of representations of the group $G$ on reflexive Banach spaces (\cite{GlasMegNewAlg} and references therein). Additionally, WAP and Roelcke compactifications of automorphism groups of $\aleph_0$-categorical structures are related to model theory (see \cite{benYaTsank} and references therein).
	
	Quite few examples of groups are known for which one can explicitly describe their Roelcke or WAP compactifications. Some of them can be found in the works \cite{Usp1, Usp2, Usp3, benYaTsank, KozLeid-2, BSorin}. For such a description, one typically constructs a certain compactification and then proves that it is equivalent to $b_rG$ or $wG$. To find suitable compactifications, the following constructions are effective: (1) Ellis compactification $e_K	G$ --- the Ellis enveloping semigroup of the action of G on some compact space K \cite{Ellis}; the graph compactification $b_X G$ of the action of $G$ on a compact or locally compact space $X$ (see, e.g., \cite{Usp2, Kennedy, Yama}).
	
	In both constructions $G$ is a topological transformation group of some space $X$. If $X$ is compact, one can directly apply approaches (1) and (2). Otherwise, one may consider equivariant compactifications of the $G$-space $X$ (if they exist) and then apply (1) or (2). This is precisely how Ellis compactifications of our group of interest are obtained (using (1)) in the article \cite{GSorin26}, and how graph compactifications are obtained (using (2)) in the present work.
	
	Approach (1) is discussed in \cite{GlasMegNewAlg, KozSor24, KozSor25}. In \cite[Corollary 4.11]{GlasMegNewAlg}, a criterion is given for when the Roelcke compactification is an Ellis compactification. A criterion for when a $G$-compactification of a group G is an Ellis compactification is given in \cite{KozLeid-1}. For the automorphism group of an ultrahomogeneous cyclically ordered set that interests us, three Ellis compactifications are described in \cite{GSorin26}. There it is also shown that none of the obtained compactifications is equivalent to the Roelcke compactification.
	
	The construction of a graph compactification of a group is as follows. Let a topological group $(G,\tau_{co})$ act effectively and continuously on a compact space $X$ ($\tau_{co}$	denotes the compact-open topology). Then, for each $g\in G$ the graph $\Gamma(g) = \{(x,g(x))\in X^2 \;\vert\; x\in X\}$ is a closed subset of $X\times X$. Let $2^{X\times X}$	denote the set of all closed subsets of $X\times X$, endowed with the Vietoris topology $\tau_V$. Then $2^{X\times X}$ is a compact space \cite{Beer}. Thus, a mapping $i_X: G\to 2^{X\times X} : g\mapsto \Gamma(g)$ is defined. The effectiveness of the action implies the injectivity of $i_X$. The restriction of the topology $\tau_V$ to $i_X	(G)$ coincides with $\tau_{co}$. By taking the closure of the image, we obtain the compactification $b_XG = \left(\ovl{i_X(G)}, i_X\right)$ of the group $G$.
	
	The same approach applies in the case of a locally compact, locally connected X, with the topology $\tau_V$	replaced by the Fell topology $\tau_F$. Also, in the case of a non-compact $X$, one may replace the $G$-space $X$ with one of its equivariant compactifications $cX$, provided that the extended action $G\acts cX$ induces the same topology on $G$: $\tau_{co}(X) = \tau_{co}(cX)$.
	
	This approach was considered in \cite{Kennedy, Usp1, Usp2, Yama, KozLeid-1, KozLeid-2}.
	
	An advantage of the graph compactification $b_X	G$ is that the left and right actions of the group on itself (by left and right multiplication) as well as the inversion operation $(\cdot)^{-1}$ extend continuously to it. The Roelcke compactification $b_rG$ has the same properties, while $wG$ possesses the first two. Therefore, $b_XG$ turns out to be a good candidate for the equalities $b_rG = b_XG$ and $wG = b_XG$. A criterion for when the graph compactification of a topological group $(G,\tau_{co})$ coincides with the Roelcke compactification is given in \cite[Theorem 3.35]{KozLeid-2} (see Lemma \ref{lemma-reolcke-equality} below).
	
	While the graph compactification may not have a semigroup structure (like the Roelcke compactification), the Ellis compactification is always a right topological semigroup.

	\bigskip
	\pagestyle{myheaders}

	In this work, we describe graph compactifications of the group $G =$ $(Aut(X,R), \tau_\partial)$ of automorphisms of a discrete ultrahomogeneous cyclically ordered set $(X,R)$. We consider the natural action $G\acts X$, where the group $G$ is endowed with the permutation topology $\tau _{\partial }$, i.e., the topology of pointwise convergence under the action on the discrete space $X$. The applicability of the described approach to the group $G$ under consideration is justified by the following facts:
	
	1) the topology of pointwise convergence is an admissible (the action is continuous) group topology on the isometry group of a metric (discrete) space \cite[Chapter 2, Ex. 3]{RD}, hence $G$ is a topological group continuously acting on $X$;
	
	2) the $G$-space $(G, X, \acts)$ is $G$-Tychonoff (its equivariant compactifications are described in \cite{GSorin26}), hence the topology of pointwise convergence $\tau_p(bX)$ for the extended action $G\acts bX$ coincides with the original topology $\tau _{\partial}$ for each equivariant compactification $bX$ \cite[Lemma 2.2]{KozSor25} or \cite[Proposition 2.10]{KozLeid-1};
	
	3) if for an action $G\acts X$ of a group on a topological space $X$ the topology of pointwise convergence $\tau _{p}$ is an admissible group topology on $G$, then $\tau_p = \tau_{co}$ \cite[Lemma 2.12]{KozLeid-2}.
	
	The main results of this work are descriptions of the Roelcke compactification $b_{r}G$ (Theorem \ref{theorem-roelcke-equals-beta}) and the WAP compactification $wG$ (Theorem \ref{theorem-wG-description}) of the group $G=(Aut(X,R),\tau_\partial)$, as well as the WAP compactification of the group $(Aut(X,R),\tau_p)$ (Theorem \ref{theorem-wap-trivial}). The Roelcke compactification is obtained as the graph compactification $b_{\beta_GX} G$ (Proposition \ref{prop-b_betaXG-description}) corresponding to the maximal equivariant compactification $\beta _{G}X$ of the $G$-space $X$. The WAP compactification is obtained as the graph compactification $b_{X}G$ (Proposition \ref{prop-b_XG-description}).

	\section{Preliminaries}\label{section-prelim}
	
	A ternary relation $R$, also denoted by $[\cdot, \cdot, \cdot]$, on a set $X$ ($|X| \geq 3$) is called a \textit{cyclic order relation} if the following conditions are satisfied:\\	
	1) \textit{cyclicity}: $[a,b,c] \implies [b,c,a]$;\\
	2) \textit{asymmetry}: $[a,b,c] \implies \lnot [b,a,c]$;\\
	3) \textit{transitivity}: $[a,b,c] \wedge [a,c,d] \implies [a,b,d]$;\\
	4) \textit{completeness of the order}: for any pairwise distinct $a,b,c \in X$, either $[a,b,c]$ or $[a,c,b]$.
	
	If $R$ satisfies these conditions, then the pair $(X,R)$ is called a \textit{cyclically ordered} (\textit{c.o.}) \textit{set} (see, e.g., \cite{Novak}).
	
	The relation $[\cdot, \cdot, \cdot]$ can be extended to all finite sets.
	
	Let $\overline{x} = (x_1,\dots, x_n)$ be an ordered tuple of points from $X$, $n\geq3$. We say that $\overline{x}$ is a \textit{cycle} in $X$ and write $[x_1,\dots, x_n]$ if $\forall\ i<j<k \; [x_i,x_j,x_k]$. We also assume that:
	
	(1) a tuple of length $n = 1$ is a cycle;
	(2) a tuple $\overline{x} = (x_1,x_2)$ of length $n = 2$ is a cycle if $x_1\neq x_2$.
	
	In particular, $[x_1,\dots, x_n]$ means that the points $x_1,\dots, x_n$ are pairwise distinct.
	
	For a point $x\in X$, a linear order $<_x$ on the set $X\setminus \{x\}$ is defined: $a<_x b \iff [x,a,b]$.
	
	For points $a,b$ in a c.o. set $(X,R)$, the \textit{interval} $(a,b)_R = \{x\in X \;\vert\; [a,x,b]\}$ is defined. By analogy with linearly ordered sets, one can define half-open intervals $[a,b)_R$, $(a,b]_R$ and the segment $[a,b]_R$. The index $R$ in the notation of intervals will be omitted. A subset $A\subset X$ inherits the cyclic order relation: $R_A = A^3 \cap R$, and the pair $(A,R_A)$ is a c.o. set. A subset $M \subset X$ is called \textit{convex} if for all $a, b\in M$, at least one of the intervals $(a,b)$, $(b,a)$ is contained in $M$.

	An \textit{isomorphism} of c.o. sets $(X,R_X)$ and $(Y,R_Y)$ is a bijection $f:X\to Y$ preserving the cyclic order: $\forall a,b,c\in X \; [a,b,c]_X \iff [fa,fb,fc]_Y$. An \textit{automorphism} of a c.o. set $(X,R)$ is an isomorphism $f:X\to X$. The set of all automorphisms of $(X,R)$ forms a group, denoted by $Aut(X,R)$. For $A,B\subset X$, a \textit{partial automorphism of $(X,R)$} is an isomorphism $p: A\to B$. The set of all partial automorphisms of $(X,R)$ forms an inverse semigroup, denoted by $pAut(X,R)$.
	
	For a subgroup $G$ of $Aut(X,R)$, the action $G\acts (X,R)$ of the group $G$ on the c.o. set $X$ is called \textit{ultratransitive} if for every $n\in \mathbb{N}$ and for any cycles $\overline{x} = (x_1, \dots, x_n)$, $\overline{y} = (y_1, \dots, y_n)$, there exists $g\in G$ such that $gx_i = y_i$, $i=1,\dots, n$. A c.o. set $(X,R)$ is called \textit{ultrahomogeneous} if the action $Aut(X,R)\acts (X,R)$ is ultratransitive. If the set $(X,R)$ is ultrahomogeneous, then the order $R$ is \textit{dense} --- $(a,b)_R\neq \emptyset$ when $a\neq b$. If $Aut(X,R)\acts (X,R)$ is ultratransitive and $x\in X$, then the action of the stabilizer of the point $x$ on the linearly ordered set $St_x\acts (X\setminus \{x\}, <_x)$ is ultratransitive, i.e., for every $n\in \mathbb{N}$ and any tuples $x_1<_x \dots <_x x_n$ and $y_1<_x \dots <_x y_n$ (where $x_i, y_i\in X\setminus\{x\}$) there exists $g\in St_x$ such that $gx_i=y_i, \; i =1,\dots, n$.
	
	Examples of ultrahomogeneous c.o. sets include the circle $\mathbb{T} = \mathbb{R}/\mathbb{Z}$ (the cyclic order is defined by choosing the counter-clockwise direction) and its countable subset $\mathbb{Q}_0 = \mathbb{Q}/\mathbb{Z}$.
	
	\medskip
	
	A $G$-space is a triple $X = (G,X,\acts)$, where $G$ is a topological group, $X$ is a topological space, and $\acts:G\times X\to X$ is a continuous action of $G$ on $X$. In particular, the group $G$ is a $G$-space, where the action is left multiplication ($(g,h)\mapsto gh$).
	
	A compactification $bX$ of a $G$-space $(G,X,\acts)$ is called {\it equivariant} if the action $\acts$ extends continuously to an action of $G$ on $bX$. For more details, see \cite{Meg84}.
	\medskip
	
	For a space $Y$, $CL(Y)$ denotes the set of all non-empty closed subsets of $Y$, and $2^Y = CL(Y) \cup \{\em\}$. For $E\subset Y$, let
	\[
	\begin{aligned}
		E^+ &= \{F\in CL(Y) \;\vert\; F\subset E\}\\
		E^- &= \{F\in CL(Y) \;\vert\; F\cap E\neq \emptyset\}.
	\end{aligned}
	\]
	
	A subbase for the \textit{Vietoris topology} $\tau_V$ on $CL(Y)$ consists of all sets of the form $W^-$ and $W^+$, where $W$ is open in $Y$. The point $\em$ is an isolated point in $2^Y$. A base for the topology $\tau_V$ on $CL(Y)$ consists of sets of the form
	\[
	[V_1, \dots, V_n] = \{F\in CL(Y) \;\vert\; \forall i\leq n \; (F\cap V_i)\neq \em, \; F\subset \bigcup\limits_{i=1}^n V_i \},
	\]
	where $V_i$ are open sets in $Y$. For brevity, instead of $[V_1, \dots, V_n]$ we will also write $[\mathcal{W}]$, where $\mathcal{W} = \{V_1, \dots, V_n\}$. The space $CL(Y) = (CL(Y),\tau_V)$ is Tychonoff if and only if $Y$ is normal. If $Y$ is compact, then $CL(Y)$ is compact.
	
	\smallskip
	
	A subbase for the \textit{Fell topology} $\tau_F$ on $CL(Y)$ consists of all sets of the form $W^-$, where $W$ is a non-empty open subset of $Y$, and all sets of the form $W^+$, where $W$ is a non-empty open subset of $Y$ with compact complement. A base for the topology $\tau_F$ at the point $\em$ consists of sets of the form $\{F\in 2^Y \;\vert\; F\cap K = \em\}$, where $K$ is a compact subset of $Y$.
	
	Thus, a base for the topology $\tau_F$ on $2^Y$ consists of sets of the form
	\[
	[V_1, \dots, V_n; K_1, \dots, K_m] = \{F\in 2^Y \;\vert\; \forall i\leq n,\; j\leq m \; (F\cap V_i) \neq\em, \; (F\cap K_j) = \em \}.
	\]
	
	In the case when one of the families $\{V_1, \dots, V_n\}$ or $\{K_1, \dots, K_m\}$ is empty, we denote the corresponding base element by $[\em; K_1, \dots, K_m]$ or $[V_1, \dots, V_n;\em]$, respectively.
	
	If $Y$ is Hausdorff and locally compact, then $2^Y$ is Hausdorff and compact. If $Y$ is compact, then $\tau_F = \tau_V$ (on the spaces $CL(Y)$ and $2^Y$). Detailed information about the topologies $\tau_V$ and $\tau_F$ can be found in \cite{Beer}.
	
	\medskip
	
	We will be interested in the spaces $2^Y$ and $CL(Y)$ in the case $Y=X\times X$. For elements $A,B\in 2^{Y}$, points $x, y, z\in X$ and $M\subset X$, we define: $sA = \{(y,x)\in Y \;\vert\; (x,y)\in A\}$; $A[x] = \{y\in X \;\vert\; (x,y)\in A\}$; $A[M] = \bigcup\limits_{x\in M} A[x]$; $D(A) = \{x\in X \;\vert\; \exists y\in X : (x,y)\in A\}$ --- the domain of $A$; $I(A) = \{y\in X \;\vert\; \exists x\in X : (x,y)\in A\} = A[D(A)]$ --- the range of $A$; $BA = \{(x,z)\in Y\;\vert\; \exists y\in X : (x,y)\in A, (y,z)\in B\}$ --- the composition of $A$ and $B$.

	\medskip
	
	A topological space $S$ that is a semigroup is called a \textit{right topological semigroup} (respectively, a \textit{semitopological semigroup}) if $\forall f,g\in S$ the map $f\mapsto fg$ is continuous (respectively, if the map $(f, g)\mapsto fg$ is separately continuous ($\forall f,g\in S$, the maps $f\mapsto fg$ and $f\mapsto gf$ are continuous)). Detailed information on topological groups and semigroups can be found in~\cite{ArhangelskiiTkachenko}.
	
	\medskip
	
	Let $G$ be a topological group. We denote by $C_b(G)$ the $C^*$-algebra of continuous bounded complex-valued functions $G\to \mathbb{C}_b$. According to the Gelfand duality, there is a correspondence between $C^*$-subalgebras of $C_b(G)$ and (improper) compactifications of $G$ (see, e.g., \cite{JStone}). We are interested in the Roelcke compactification $b_rG$ and the WAP compactification $wG$, which can be defined as follows.
	
	The \textit{Roelcke compactification} $b_rG$ is the Gelfand space (the space of maximal ideals) of the $C^*$-algebra $UC(G)$ of all uniformly continuous functions $(G,\mathcal{L}\land\mathcal{R}) \to \mathbb{C}$, where $\mathcal{L}\land\mathcal{R}$ denotes the Roelcke uniformity on the group $G$ --- the infimum of the left uniformity $\mathcal{L}$ and the right uniformity $\mathcal{R}$, and $\mathbb{C}$ is endowed with the natural uniformity. Equivalently, $b_rG$ is the Samuel compactification of the uniform space $(G,\mathcal{L}\land\mathcal{R})$. A base for the Roelcke uniformity $\mathcal{L}\land\mathcal{R}$ on the group $G$ consists of coverings $\{OgO \;\vert\; g\in G\}$, where $O$ are neighbourhoods of the identity in $G$. Detailed information on uniform structures can be found in \cite{RD}.
	
	For a function $\phi\in C_b(G)$, we define its orbit $Orb(\phi) = \{\phi_g\;\vert\;g\in G\}$, where $\phi_g(h) = \phi(gh)$. A function $\phi\in C_b(G)$ is called \textit{weakly almost periodic} if its orbit $Orb(\phi)$ is a relatively weakly compact subset of $C_b(G)$. This means that the closure of the set $Orb(\phi)$ in the weak topology of the Banach space $C_b(G)$ is compact \cite{BergJungMiln}.
	
	The \textit{WAP compactification} $wG$ is the Gelfand space of the $C^*$-algebra $WAP(G)$ of all weakly almost periodic functions on $G$. While the Roelcke compactification is always proper ($G$ is homeomorphically densely embedded in $b_rG$), the WAP compactification may be improper ($G$ is continuously homomorphically mapped onto a dense subset of $wG$). The compactification $wG$ is the largest in the class of improper semitopological semigroup compactifications of $G$. However, in the case of interest to us --- the action of $G$ on a discrete $X$ --- the Ellis compactification $e_{\alpha X}G$ is a semitopological semigroup compactification \cite[Theorem 1]{GSorin26} ($\alpha X$ denotes the Alexandroff one-point compactification). That is, $G$ has a proper semitopological semigroup compactification, hence the compactification $wG$ ($\geq e_{\alpha X}G$) is proper.
	
	\medskip
	
	Topological spaces are assumed to be Tychonoff, and by an action we mean an effective left action. Compactifications are proper. In the notation of compactifications $(bX,b)$, we will omit the embedding $b:X\to bX$ and write simply $bX$. Equivalence of compactifications is denoted by the symbol $=$. In the case when two equivariant compactifications of a topological group $G$ have the same algebraic structure, the map between these compactifications will preserve this structure \cite[\S 3.3]{KozLeid-1}.
	
	Further, $X = (X,R)$ is an infinite ultrahomogeneous c.o. set in the discrete topology, and $G = (Aut(X,R),\tau_\partial)$ is its automorphism group in the permutation topology. As noted at the end of Section~\ref{section-intro}, $G$ is a topological group acting continuously on $X$.
	
	We follow the terminology from \cite{Engelking}. The closure of a set $A$ is denoted by $\overline{A}$. When considering graph compactifications of the form $b_XG = \left( \overline{i_X(G)}, i_X \right)$, for brevity instead of $i_X(g)$ we sometimes write $\Gamma(g)$ --- the graph of an element $g$ --- a closed subset of $X\times X$.
	
	For a uniform space $(Y,\mathcal{U})$, capital letters $U,V$ denote entourages of the diagonal (elements of the uniformity $\mathcal{U}$), and lowercase letters $u,v$ denote uniform coverings. The symbol $\mathcal{U}$ simultaneously denotes the family of entourages of the diagonal and the family of uniform coverings.

	\section{The Roelcke compactification of the group \texorpdfstring{$Aut(X,R)$}{Aut(X,R)}}

	To describe $b_rG$, we need $\beta_GX$ --- the maximal equivariant compactification of the $G$-space $X$.

	\begin{propos}\textup{\cite[Proposition 1]{GSorin26}}\label{prop-beta-G-X}
		$\beta_G X = c(X\otimes_l \{-1,0,1\})$.
	\end{propos}
	
	By $X\otimes_l \{-1,0,1\}$ we denote the lexicographic product of the c.o. set $X$ and the linearly ordered set $-1<0<1$ --- each point $x\in X$ is replaced by three consecutive points $x^-, x, x^+$, where $x^- \in X^- = X\times \{-1\}$, $x \in X = X\times \{0\}$, and $x^+ \in X^+ = X\times \{1\}$. The set $X\otimes_l \{-1,0,1\}$ is cyclically ordered. The c.o. set $c(X\otimes_l \{-1,0,1\})$ is obtained from $X\otimes_l \{-1,0,1\}$ by filling each gap with a single point. The space $\beta_G X$ is a cyclically ordered compactum \cite[Theorem 2.19]{Megreli23}; a base for its topology is formed by all possible intervals. The action $G\acts \beta_GX$ satisfies the properties $g(x^-) = (gx)^-$ and $g(x^+) = (gx)^+$ for all $x\in X$. For more details, see \cite{GSorin26, GSorin24}.
	
	\begin{lemma}\label{lemma-b-betaXG-necess-cond}
		Let $A\in b_{\beta_GX} G$. Then $A$ satisfies the following properties:
		\begin{enumerate}
			\item[1)] $\forall x\in \beta_GX \; (A[x] \neq \em \neq sA[x])$;
			\item[2)] $\forall x\in X \; (|A[x]| = 1 = |sA[x]|)$;
			\item[3)] \textit{Monotonicity}: $\forall x_1, x_2, x_3, y_1, y_2, y_3 \in \beta_GX : y_i \in A[x_i]\, (i=1,2,3) \; (x_1 \neq x_2 \neq x_3 \neq x_1 \land y_1\neq y_2 \neq y_3 \neq y_1 \implies ([x_1,x_2,x_3] \iff [y_1,y_2,y_3]))$;
			\item[4)] $\forall x_1 \neq x_2\in \beta_GX \; \forall y_1\in A[x_1], y_2\in A[x_2] \; ([y_1, y_2]\subset A[[x_1,x_2]])$.
		\end{enumerate}
	\end{lemma}
	
	\begin{proof}
		An embedding $i_{\beta_GX} : G\to CL(\beta_GX\times \beta_GX)$ is defined. Let $A \in \ovl{i_{\beta_GX}(G)} = b_{\beta_GX} G$. We prove that properties 1--4 hold.
		
		Property 1 holds for graph compactifications; see \cite{Usp2}.
		
		2) Suppose $\exists x\in X : |A[x]| \geq 2$. Then there exist distinct $y, z\in A[x]$. There exist neighbourhoods $O_y\ni y$ and $O_z \ni z$ such that $O_y \cap O_z = \em$. For $W_y = \{x\}\times O_y$ and $W_z = \{x\}\times O_z$, we have $A\in W_y^- \cap W_z^-$. Since $A \in \ovl{i_{\beta_GX}(G)}$, there exists $g\in G$ such that $\Gamma(g) \in W_y^- \cap W_z^-$, i.e., $g(x)\in O_y\cap O_z \neq \em$. This is a contradiction.
		
		3) Let $y_i \in A[x_i], (i=1,2,3)$, and $x_1 \neq x_2 \neq x_3 \neq x_1 \land y_1\neq y_2 \neq y_3 \neq y_1$. We prove $[y_1,y_2,y_3] \implies [x_1,x_2,x_3]$.
		
		Suppose the contrary: $[y_1,y_2,y_3] \land \lnot [x_1,x_2,x_3]$. Since $x_1 \neq x_2 \neq x_3 \neq x_1$, we have $[x_1,x_3,x_2]$. Then there exist pairwise disjoint convex neighbourhoods $O_{x_1}, O_{x_2}, O_{x_3}$ of the points $x_1,x_2,x_3$, respectively. Moreover, for any $t_i \in O_{x_i}, (i=1,2,3)$, we have $[t_1,t_3,t_2]$. Similarly, for the points $y_1,y_2,y_3$, we choose pairwise disjoint convex neighbourhoods $O_{y_1}, O_{y_2}, O_{y_3}$. Then, since $[y_1,y_2,y_3]$, for any $r_i \in O_{y_i}, (i=1,2,3)$, we have $[r_1,r_2,r_3]$. For each $i = 1,2,3$, we have $A\cap (O_{x_i}\times O_{y_i}) \neq \em$. Since $A \in \ovl{i_{\beta_GX}(G)}$, there exists $g\in G$ such that $\forall i =1,2,3 \;\; \Gamma(g) \cap (O_{x_i}\times O_{y_i}) \neq \em$. Consequently, there exist points $t_i \in O_{x_i}\cap X, (i=1,2,3)$ such that $gt_i \in O_{y_i}, (i=1,2,3)$. Then $[t_1,t_3,t_2]$ and $[gt_1,gt_3,gt_2]$ (since $g$ is an automorphism). This contradicts $[r_1,r_2,r_3]$ for any $r_i \in O_{y_i}, (i=1,2,3)$.
		
		The statement $[x_1,x_2,x_3] \implies [y_1,y_2,y_3]$ is proved similarly.

		4) Suppose there exists $z\in (y_1, y_2)$ such that $z\notin A[[x_1, x_2]]$. By property 1, there exists $p\in \beta_GX\sm [x_1, x_2] = (x_2, x_1)$ such that $A[p] \ni z$. Then $[x_1, x_2, p]$ and $[y_1, z, y_2]$ contradict property 3.
	\end{proof}

	\begin{remark}
		Every $A\in b_{\beta_GX} G$ satisfies the property:
		5) $\forall t\in \beta_GX\sm X \; \forall x\in X \; (x\in A[t] \implies x^-, x^+ \in A[t]) \land (x\in sA[t] \implies x^-, x^+ \in sA[t])$.
	\end{remark}

	\begin{propos}\label{prop-b_betaXG-description}
		The compactification $b_{\beta_GX} G$ is a subspace\\
		$\{A\in CL(\beta_GX\times \beta_GX) \;\vert\; A \text{ satisfies properties 1--3}\}$ of the space $(CL(\beta_GX\times \beta_GX), \tau_V)$.
	\end{propos}
	
	\begin{proof}
	By Lemma~\ref{lemma-b-betaXG-necess-cond}, every element of $b_{\beta_GX} G$ satisfies properties 1--4. We show that every element of $CL(\beta_GX\times \beta_GX)$ satisfying properties 1--4 lies in $b_{\beta_GX} G = \ovl{i_{\beta_GX} (G)}$.
		
	Let $A \in [\W] \in \tau_V$. Since $\dim \beta_GX = 0$, we may assume that $\W = \{W_1, \dots, W_m\}$, where $W_i = U_j\times U_k$. Here $U_j, U_k \in u = \{(a_t, a_{t+1}), \{a_t\} \}_{t=1}^n$, where $(a_1, \dots, a_n)$ is a cycle in $X$ (we may assume $n\geq 3$). The family $\W$ consists of all elements of the disjoint covering $u^2 = \{U_i\times U_j\;\vert\; U_i, U_j\in u\}$ that intersect $A$.

	Each element of $u^2$ has one of four types: a point $\{a_i\}\times \{a_j\}$; a vertical interval $\{a_i\}\times (a_j, a_{j+1})$; a horizontal interval $(a_i, a_{i+1})\times \{a_j\}$; a rectangle $(a_i, a_{i+1})\times (a_j, a_{j+1})$.
	
	We renumber the elements of $\W$ as follows:
	
	a) assign number 1 to the element of $\W$ of the form $\{a_1\}\times V$ for some $V\in u$. By property 1, such an element exists; by property 2, it is unique. $W_1 = \{a_1\}\times V$;
	
	b) if number $k$ has already been assigned, assign number $k+1$:\\	
	b1) if $W_k = \{a_i\}\times (a_j, a_{j+1})$, then $W_{k+1} = (a_i, a_{i+1})\times (a_j, a_{j+1})$;\\	
	b2) if $W_k = \{a_i\}\times \{a_j\}$, then $W_{k+1} = (a_i, a_{i+1})\times (a_j, a_{j+1})$;\\	
	b3) if $W_k = (a_i, a_{i+1})\times \{a_j\}$, then $W_{k+1} = (a_i, a_{i+1})\times (a_j, a_{j+1})$;\\	
	b4) if $W_k = (a_i, a_{i+1})\times (a_j, a_{j+1})$ and $\{a_{i+1}\}\times \{a_{j+1}\} \in \W$, then $W_{k+1} = \{a_{i+1}\}\times \{a_{j+1}\}$;\\	
	b5) if $W_k = (a_i, a_{i+1})\times (a_j, a_{j+1})$ and exactly one of the intervals $\{a_{i+1}\}\times (a_j, a_{j+1})$, $(a_i, a_{i+1})\times \{a_{j+1}\}$ lies in $\W$, then let $W_{k+1}$ be this interval ($\in \W$);\\	
	b6) if $W_k = (a_i, a_{i+1})\times (a_j, a_{j+1})$ and both intervals $\{a_{i+1}\}\times (a_j, a_{j+1})$, $(a_i, a_{i+1})\times \{a_{j+1}\}$ lie in $\W$, then set $W_{k+1} = (a_i, a_{i+1})\times \{a_{j+1}\}$ if this interval has not yet been numbered, and $W_{k+1} = \{a_{i+1}\}\times (a_j, a_{j+1})$ otherwise. In case (b6), the set $(a_i, a_{i+1})\times (a_j, a_{j+1})$ will be numbered twice.
	
	(Example: $A = \{z\} \times \beta_GX \cup \beta_GX\times \{z\}$ for some $z\in \beta_GX\sm X$.)
	
	The numbering ends when $W_{k+1} = W_1$ (the set $W_1$ is not numbered a second time).
	
	\medskip
	
	\textbf{Justification}
	
	\textbf{I.} At each step, the numbering is unambiguous, $(W_k\in \W \implies W_{k+1} \in \W)$, and the process terminates.
	
	Case (b1). Let $W_k = \{a_i\}\times (a_j, a_{j+1})$. Then $W_{k+1} = (a_i, a_{i+1})\times (a_j, a_{j+1}) \in \W$, since otherwise $A[a_i^+] \cap (a_j, a_{j+1}) = \em$. This contradicts property 4 (since the interval $(A[a_i], a_{j+1}) \neq \em$ must be contained in $A[a_i^+]$).
	
	Similarly, in cases (b2) and (b3), we have $W_{k+1}\in \W$.
	
	Cases (b4), (b5), (b6). Let $W_k = (a_i, a_{i+1})\times (a_j, a_{j+1})$, so $\exists (x,y)\in A \cap W_k$. Suppose $A$ does not intersect any of the sets $\{a_{i+1}\}\times \{a_{j+1}\}$, $\{a_{i+1}\}\times (a_j, a_{j+1})$, $(a_i, a_{i+1})\times \{a_{j+1}\}$. By property 1, $\exists b\in A[a_{i+1}]$, $\exists c\in sA[a_{j+1}]$. Then $c\in (a_{i+1}, a_i]$, $x\in (a_i, a_{i+1})$, hence $[x, a_{i+1}, c]$. Similarly, $[y, a_{j+1}, b]$. By property 3, $[x, a_{i+1}, c]$ implies $[y, b, a_{j+1}]$. Contradiction. Thus, at least one of cases (b4), (b5), (b6) holds and $W_{k+1}\in \W$.
	
	In case (b4), by property 2, cases (b5) and (b6) cannot occur. Therefore, when $W_k = (a_i, a_{i+1})\times (a_j, a_{j+1})$, exactly one of cases (b4), (b5), (b6) occurs.

	In case (b6), we have $A[(a_i, a_{i+1})] \supset \beta_GX\sm (a_j, a_{j+1})$ and $sA[(a_j, a_{j+1})] \supset \beta_GX\sm (a_i, a_{i+1})$. Indeed, otherwise, by property 1, there exists $(x, y)\in A$ such that $x\notin (a_i, a_{i+1})$ and $y\notin (a_j, a_{j+1})$. By property 2, $x\neq a_{i+1}$ and $y\neq a_{j+1}$. There exist $b\in A[a_{i+1}]\cap (a_j, a_{j+1})$ and $c\in sA[a_{j+1}]\cap (a_i, a_{i+1})$. Then $[b, a_{j+1}, y]$ and $[c, a_{i+1}, x]$. From the latter, by property 3, we get $[a_{j+1}, b, y]$. This contradiction proves the inclusions.
	
	Therefore, in case (b6), the family $\W$, together with the rectangle $W_k = (a_i, a_{i+1})\times (a_j, a_{j+1})$, contains all four adjacent intervals and all sets of the form $(a_i, a_{i+1})\times V$, where $V\in u$. Then the rectangle will be assigned two numbers: $k$ and $k+2n$ (the covering $u$ contains $2n$ elements).
	
	At each step, it is possible to proceed to the next one. Since $\W$ is finite and each element is numbered at most twice, the process will terminate.
	
	\textbf{II.} All elements of $\W$ are numbered.
	
	Suppose there exists $U\times V\in \W$ that was not numbered. From the numbering algorithm, it follows that there exist $U', V'\in u$ such that the sets $U\times V'$ and $U'\times V$ (in $\W$) have been numbered. If $U$ or $V$ is of the form $\{a_i\}$, we obtain a contradiction with property 2.
	
	If $U\times V = (a_i, a_{i+1})\times (a_j, a_{j+1})$, then, arguing similarly to part I, one can show that one of the preceding sets --- $\{a_i\}\times (a_j, a_{j+1})$, $(a_i, a_{i+1})\times \{a_j\}$, or $\{a_i\}\times \{a_j\}$ --- also lies in $\W$. Due to the unambiguity of each step, this preceding set was also not numbered. This contradicts what was proven above.
	
	Thus, every element of $\W$ is numbered, and the rectangle from case (b6) (if it exists) is numbered twice. The number of rectangles from (b6) is $\leq 1$ (this follows from the inclusions $A[(a_i, a_{i+1})] \supset \beta_GX\sm (a_j, a_{j+1})$ and $sA[(a_j, a_{j+1})] \supset \beta_GX\sm (a_i, a_{i+1})$ and property 2).
	
	\textbf{Existence of $g\in G$ with the property $\Gamma(g) \in [\W]$.}
	
	Let the resulting numbering of the family $\W$ be $W_1, \dots, W_l$. We show that for each $i=1, \dots, l$, there exists $(x_i, y_i)\in X^2\cap W_i$, such that $[x_1, \dots, x_l]$ and $[y_1, \dots, y_l]$ (if some $W_k$ is numbered twice, then two points will be chosen in it).
	
	Indeed, $W_k = U_k\times V_k$ is a product of intervals. Set $x_1 = a_1 \in U_1 = \{a_1\}$, and either $y_1 = a_j$ if $V_1 = \{a_j\}$, or $y_1 \in X\cap V_1 = X\cap (a_j, a_{j+1})$ --- an arbitrary point. We have $U_2 = (a_1, a_2)$, $V_2 = (a_j, a_{j+1})$. Then choose $x_2 \in U_2 \cap X$ and $y_2 \in (y_1, a_{j+1})\cap X$.
	
	Each subsequent $x_i$ is an arbitrary element of the set $U_i \cap (x_{i-1}, x_1) \cap X$, and $y_i \in V_i \cap (y_{i-1}, y_1)\cap X$ (such elements exist due to the density of the order).

	Due to the ultra-transitivity of the action $G\acts (X,R)$, there exists $g\in G$ such that $gx_i = y_i$ for $i=1, \dots, l$. Consequently, $\forall i \; \Gamma(g)\cap W_i \neq \em$.
	
	Also, $\Gamma(g) \subset \bigcup\limits_{i} W_i$. Indeed, otherwise there exists $U\times V \in u^2$ such that $\Gamma(g) \cap U\times V\neq \em$ and $U\times V \notin \W$. Since $\Gamma(g)$ satisfies properties 1--4, applying the reasoning from part II to it yields a contradiction with property 2.

	Thus, $\Gamma(g)\in [\W]$.
\end{proof}

	We proceed to prove the equality $b_rG = b_{\beta_GX}G$. We use the following
	
	\begin{lemma}\textup{\cite[Theorem 3.35]{KozLeid-2}}\label{lemma-reolcke-equality}
		Let the action $(G, \tau_{co})\acts K$ of a topological group on a compact space $K$ be continuous. Then the following conditions are equivalent:
		\begin{enumerate}
			\item[\textup{1.}] $\begin{aligned}[t]
				&\forall U\in \U_K \;\exists V\in \U_K : (i_K(f), i_K(h)) \in CL(V\times V), \; f,h\in G \implies &\\
				&\exists g\in G : \forall x\in K \; (f(x),g(x))\in U, \; (g^{-1}(x), h^{-1}(x)) \in U &
			\end{aligned}$
			\item[\textup{2.}] $G$ is Roelcke-precompact and $b_rG = b_KG$.
		\end{enumerate}
	\end{lemma}
	
	Here $\U_K$ is the unique uniformity on the compact space $K$; $U, V$ are entourages of the diagonal in $K\times K$; $CL(V\times V) = \{(A,B) \in CL(K\times K)\times CL(K\times K) \;\vert\; B\subset st(A,v), \; A\subset st(B,v) \}$ is an entourage of the diagonal in $CL(K\times K)\times CL(K\times K)$ generated by $V$; $v$ denotes the uniform covering corresponding to the entourage $V$, i.e. $v = \{ V[x] \;\vert\; x\in K \}$, where $V[x] = \{y\in K \;\vert\; (x,y)\in V\}$; $st(A,v)$ denotes the star of the set $A$ with respect to the covering $v$.
	
	The family of entourages $\{CL(V\times V) \;\vert\; V\in \U_K\}$ forms a base of the unique uniformity on $CL(K\times K)$ (see \cite[Problem 8.5.16]{Engelking}).
	
	In terms of uniform coverings, condition 1 of Lemma~\ref{lemma-reolcke-equality} takes the form:
	\begin{align*}\label{condition-1}
		&\forall u\in \U_K \;\exists v\in \U_K : \forall f,h\in G: (\Gamma(f),\Gamma(h))\in CL(v\times v) \; \exists g\in G:\\ \tag{$\star$}
		&\text{the following conditions hold}\\  
		&{\rm I)} \; \forall x\in K \; fx, gx \text{ belong to the same element of the covering } u,\\  
		&{\rm II)} \; \forall x\in K \; g^{-1}x, h^{-1}x \text{ belong to the same element of the covering } u.
	\end{align*}
	
	\begin{theorem}\label{theorem-roelcke-equals-beta}
		$b_rG = b_{\beta_GX}G$.
	\end{theorem}
	\begin{proof}
		We verify that condition $(\star)$ is satisfied.
		
		The base $\U_{\beta_GX}$ consists of coverings of the form $u_{\ovl{x}} = \{ \{x_i\}, (x_i, x_{i+1}) \}_{i=1}^n$, where $\ovl{x} = (x_1, \dots, x_n)$ is a cycle in $X$ and $n\geq 3$ \cite[Lemma~1]{GSorin26}. We number the elements of $u_{\ovl{x}}$ in an arbitrary order: $u_{\ovl{x}} = \{U_i \;\vert\; i=1,\dots, 2n\}$.
		
		The base of the uniformity on $\beta_GX\times \beta_GX$ consists of coverings $u_{\ovl{x}}^2 = \{U_i\times U_j \;\vert\; U_i, U_j\in u_{\ovl{x}} \}$.
		
		The base of the uniformity on $CL(\beta_GX\times \beta_GX)$ consists of entourages of the form $CL(u_{\ovl{x}}^2) = \{ (A,B) \in CL(\beta_GX\times \beta_GX)^2 \;\vert\; B\subset st(A,u_{\ovl{x}}^2); A\subset st(B,u_{\ovl{x}}^2) \}$.
		
		Let $u$ in condition $(\star)$ be $u_{\ovl{x}}$. Set $v := u$. Fix arbitrary $f, h\in G$. The condition $(\Gamma(f),\Gamma(h))\in CL(v\times v)$ becomes $(\Gamma(f),\Gamma(h))\in CL(u_{\ovl{x}}^2)$, i.e. $\Gamma(f)\subset st(\Gamma(h),u_{\ovl{x}}^2)$ and $\Gamma(h)\subset st(\Gamma(f),u_{\ovl{x}}^2)$.
		
		Suppose $\Gamma(f)\subset st(\Gamma(h),u_{\ovl{x}}^2)$. Assume $\Gamma(f)\cap (U_i\times U_j) \neq \emptyset$. Since $u_{\ovl{x}}^2$ is disjoint, $\Gamma(h)\cap (U_i\times U_j) \neq \emptyset$. Then from the conditions $\Gamma(f)\subset st(\Gamma(h),u_{\ovl{x}}^2)$ and $\Gamma(h)\subset st(\Gamma(f),u_{\ovl{x}}^2)$ it follows that
		\begin{equation}\label{f-h-property}
			\Gamma(f) \text{ and } \Gamma(h) \text{ intersect the same sets of the form } U_i\times U_j, \tag{$\ast$}
		\end{equation}
		where $U_i, U_j\in u_{\ovl{x}}$.
		
		Now we find an element $g\in G$ required by condition $(\star)$.\\
		Consider the set $\{x_1, \dots, x_n, f^{-1}x_1, \dots, f^{-1}x_n\}$. It has cardinality $m$, with $n\leq m\leq 2n$. Denote its elements by $b_1, \dots, b_m$ such that $[b_1, \dots, b_m]$. Define the value of $g$ at points $b_1, \dots, b_m$ according to the following rules:\\
		1) if $b_i = x_j$, then $gb_i = gx_j := hx_j$;\\
		2) if $b_i = f^{-1}x_j$, then $gb_i = gf^{-1}x_j := x_j$.\\
		If $b_i$ falls under both rules --- $b_i = x_j = f^{-1}x_k$, then $fx_j = x_k$. Since $\{x_j\}, \{x_k\} \in u_{\ovl{x}}$, by property (\ref{f-h-property}) we have $hx_j = x_k$. From rule 1) we get $gb_i = gx_j = hx_j = x_k$. From rule 2) we get $gb_i = gf^{-1}x_k = x_k$. The rules are consistent, and the definition is correct.

		We now show that $[gb_1, \dots, gb_m]$, i.e. for all $1\leq i<j<k\leq m$ we have $[gb_i, gb_j, gb_k]$. We have $[b_i, b_j, b_k]$.
		
		Case 1. $b_i, b_j, b_k$ are defined by rule 1). Then $b_i = x_q$, $b_j = x_r$, $b_k = x_s$, and $gb_i = hb_i$, $gb_j = hb_j$, $gb_k = hb_k$. Since $h \in G$, we have $[gb_i, gb_j, gb_k]$.
		
		Case 2. $b_i, b_j, b_k$ are defined by rule 2). Then $b_i = f^{-1}x_q$, $b_j = f^{-1}x_r$, $b_k = f^{-1}x_s$, and $gb_i = fb_i$, $gb_j = fb_j$, $gb_k = fb_k$. Since $f \in G$, we have $[gb_i, gb_j, gb_k]$.
		
		Case 3. $b_i, b_j$ are defined by rule 1), and $b_k$ is defined by rule 2). Then $b_i = x_q$, $b_j = x_r$, $b_k = f^{-1}x_s$, and $gb_i = hx_q$, $gb_j = hx_r$, $gb_k = x_s$. Since $[b_i, b_j, b_k]$, we have $f^{-1}x_s \in (x_r, x_q)$. This implies that $\exists a \in (x_r, x_q)$ such that $fa = x_s$. Then $\exists t$ such that $a \in U_t$. Consequently, $\Gamma(f)$ intersects the set $U_t \times \{x_s\}$. By property (\ref{f-h-property}), $\exists c \in U_t \subset (x_r, x_q)$ such that $hc = x_s$. Therefore, $x_s = hc \in (hx_r, hx_q)$, which implies $[hx_q, hx_r, x_s]$. That is, $[gb_i, gb_j, gb_k]$.
		
		Case 4. $b_i, b_j$ are defined by rule 2), and $b_k$ is defined by rule 1). Similarly to Case 3, we obtain $[gb_i, gb_j, gb_k]$.
		
		Thus, $[gb_1, \dots, gb_m]$. Consequently, by the ultratransitivity of the action, $\exists g \in G$ satisfying rules 1) and 2).
		
		Now we show that the following conditions are satisfied:\\
		I) $\forall z \in \beta_GX \; \exists t: fz, gz \in U_t$;\\
		II) $\forall z \in \beta_GX \; \exists t: g^{-1}z, h^{-1}z \in U_t$.
		
		Condition I). If $z = x_i$, then $gz = gx_i = hx_i \in U_t$ for some $t$. By property (\ref{f-h-property}), $fx_i \in U_t$. That is, $\exists t: fz, gz \in U_t$.
		
		If $z = f^{-1}x_i$, then $gz = x_i \in \{x_i\} \in u_{\ovl{x}}$. Also, $fz = x_i \in \{x_i\}$.
		
		Let $z \in (b_i, b_{i+1})$.\\
		Case 1. $b_i = x_q$, $b_{i+1} = x_r$. Since the interval $(b_i, b_{i+1})$ does not contain any points from the set $\{b_1, \dots, b_m\}$, it does not contain the points $x_1, \dots, x_n$, so $x_r = x_{q+1}$. Also, $(b_i, b_{i+1})$ does not contain the points $f^{-1}x_1, \dots, f^{-1}x_n$, so $\exists t$ such that $(fb_i, fb_{i+1}) = (fx_q, fx_{q+1}) \subset U_t$. Then $fz \in U_t$. From $(fx_q, fx_{q+1}) \subset U_t$, it also follows that $\Gamma(f)$ intersects $(x_q, x_{q+1}) \times U_t$ and does not intersect any set of the form $(x_q, x_{q+1}) \times U_l$ for $l \neq t$. By (\ref{f-h-property}), $\Gamma(h)$ has the same property. Consequently, $gz \in (gb_i, gb_{i+1}) = (hx_q, hx_{q+1}) \subset U_t$.\\
		Case 2. $b_i = f^{-1}x_q$, $b_{i+1} = f^{-1}x_r$. Since the interval $(b_i, b_{i+1})$ does not contain the points $f^{-1}x_1, \dots, f^{-1}x_n$, we have $f^{-1}x_r = f^{-1}x_{q+1}$. Then $fz \in (fb_i, fb_{i+1}) = (x_q, x_{q+1})$ and $gz \in (gb_i, gb_{i+1}) = (x_q, x_{q+1})$.\\
		Case 3. $b_i = x_q$, $b_{i+1} = f^{-1}x_r$. Consider the point $f^{-1}x_{r-1}$. Since $f$ is a bijection, $f^{-1}x_{r-1} \neq f^{-1}x_r$. If $f^{-1}x_{r-1} = x_q$, we obtain Case 2. We may assume that $f^{-1}x_{r-1}$ is distinct from $b_i$ and $b_{i+1}$. Since $(b_i, b_{i+1})$ does not contain the points $f^{-1}x_1, \dots, f^{-1}x_n$, we have $x_q \in (f^{-1}x_{r-1}, f^{-1}x_r)$. Therefore, $hx_q = gx_q \in (gf^{-1}x_{r-1}, gf^{-1}x_r) = (x_{r-1}, x_r)$. By (\ref{f-h-property}), $fx_q \in (x_{r-1}, x_r)$. Since $z \in (b_i, b_{i+1})$, we have $fz \in (fx_q, x_r)$, which together with $fx_q \in (x_{r-1}, x_r)$ implies $fz \in (x_{r-1}, x_r)$. Also, $gz \in (gx_q, x_r)$, and together with $gx_q \in (x_{r-1}, x_r)$, this implies $gz \in (x_{r-1}, x_r)$.\\
		Case 4. $b_i = f^{-1}x_q$, $b_{i+1} = x_r$. Similarly to Case 3, consider the point $f^{-1}x_{q+1}$.\\
		In each case, Condition I is satisfied.
		
		Condition II. This is proved similarly, using the fact that condition (\ref{f-h-property}) holds for $\Gamma(f^{-1})$ instead of $\Gamma(f)$ and for $\Gamma(h^{-1})$ instead of $\Gamma(h)$. The mapping $g^{-1}$ satisfies the rules:\\
		1') $g^{-1}x_j = f^{-1}x_j$;\\
		2') $g^{-1}hx_j = x_j$,\\
		which follow from rules 2) and 1), respectively.

		Therefore, by the lemma \ref{lemma-reolcke-equality}, $b_rG = b_{\beta_GX}G$.
	\end{proof}

	\section{The WAP compactification of the group \texorpdfstring{$Aut(X,R)$}{Aut(X,R)}}
	
	For a partial automorphism $p \in pAut(X,R)$, we denote by $D(p)$ and $I(p)$ its domain and range, respectively. We identify the graph $\Gamma(p) \subset X^2$ with the isomorphism $p: D(p) \to I(p)$. This identification is an isomorphism of semigroups with involution.
	
	\begin{propos}\label{prop-b_XG-description}
		The set $b_XG := \ovl{i_X(G)}$ is algebraically isomorphic to the semigroup $pAut(X,R)$.
		
		Moreover, $b_XG$ is a semitopological semigroup with a continuous involution $s: A \mapsto sA$.
	\end{propos}

	\begin{proof}
		Since the space $X$ is locally compact and locally connected, the graphical compactification $b_XG$ is well defined \cite[Theorem~3.33]{KozLeid-2}.
		
		Inclusion $pAut(X,R) \subset b_XG$:\\
		let $f \in pAut(X,R)$. Consider an arbitrary small neighbourhood of $f$ of the form $[\{(x_1,y_1)\}, \dots, \{(x_n,y_n)\};K]$ (if $f = \em$, the neighbourhood has the form $[\em;K]$). The set $K$ is finite. Let $pr_1,pr_2$ be the projections of $X\times X$ onto the first and second factors. Let $\{a_1,\dots, a_m\} := pr_1(K) \sm \{x_1,\dots, x_n\}$. By the density of the order $R$, the intersection of the set $X \sm pr_2(K)$ with any interval is infinite. Then, by the ultratransitivity of the action $G\acts (X,R)$, there exists $g \in G$ such that $gx_i = y_i$ and $ga_j \in X \sm pr_2(K)$. Hence, $g \in [\{(x_1,y_1)\}, \dots, \{(x_n,y_n)\};K]$.
		
		Inclusion $pAut(X,R) \supset b_XG$:\\
		let $A \in b_XG$. We have $\em \in b_XG$ and $\em \in pAut(X,R)$. Assume that $A \neq \em$. Consequently, $D(A) \neq \em \neq I(A)$. Then $\forall x \in D(A) \; \exists ! y \in I(A) : (x,y) \in A$. Indeed, the existence follows from the definition of $D(A)$. Let $y_1, y_2 \in A[x]$. Then $A \in W = \{(x,y_1)\}^- \cap \{(x,y_2)\}^-$. If $y_1 \neq y_2$, we get a contradiction, since $W \cap i_X(G) = \em$. Similarly, $\forall y \in I(A) \; \exists ! x \in D(A) : (x,y) \in A$. That is, $A$ is the graph of a bijection $D(A) \to I(A)$.

		We have $\forall x_1, x_2, x_3 \in D(A) \; ([x_1,x_2,x_3] \implies [A[x_1], A[x_2], A[x_3]])$. Indeed, $[x_1,x_2,x_3]$ implies that $x_1, x_2, x_3$ are distinct. By the above, $A[x_1], A[x_2], A[x_3]$ are also distinct. Suppose $\lnot [A[x_1], A[x_2], A[x_3]]$. Then $[A[x_1], A[x_3], A[x_2]]$. We have $A \in W = \{(x_1, A[x_1])\}^- \cap \{(x_2, A[x_2])\}^- \cap \{(x_3, A[x_3])\}^-$. Then if $g \in i_X(G) \cap W$, we have $gx_i = A[x_i]$. Hence $[gx_1, gx_3, gx_2]$, which implies $[x_1, x_3, x_2]$, contradicting $[x_1,x_2,x_3]$. Thus, $i_X(G) \cap W = \em$, a contradiction with $A \in b_XG$. Therefore, $A$ is the graph of some isomorphism $(D(A),R) \to (I(A),R)$.

		Identifying the sets $b_XG$ and $pAut(X,R)$, we verify the continuity of left multiplication. Fix $p,q \in b_XG$ and a neighbourhood\\
		$W = [\{(x_1,y_1)\}, \dots, \{(x_n,y_n)\}; \{(s_1,t_1)\}, \dots, \{(s_m,t_m)\}] \ni pq$. Then $q \in U = [\{(x_1,p^{-1}y_1)\}, \dots, \{(x_n,p^{-1}y_n)\};\em]$. For all $i=1,\dots, m$ such that $s_i \in D(q)$, let $V_i = [\{(s_i,q(s_i))\};\em]$. For the remaining $i$, either $t_i \in I(p)$, then let $V_i = [\em;\{(s_i,p^{-1}(t_i))\}]$, or $t_i \notin I(p)$, then let $V_i = b_XG$. Thus, $q \in U \cap \bigcap\limits_{i=1}^m V_i$ and $pq' \in W$ for every $q' \in U \cap \bigcap\limits_{i=1}^m V_i$. That is, multiplication by $p$ on the left is continuous at the point $q$. Continuity of right multiplication is verified similarly.
		
		Continuity of $s$. Let $W = [\{(x_1,y_1)\}, \dots, \{(x_n,y_n)\}; \{(s_1,t_1)\}, \dots, \{(s_m,t_m)\}]$. Then $sW = [\{(y_1,x_1)\}, \dots, \{(y_n,x_n)\}; \{(t_1,s_1)\}, \dots, \{(t_m,s_m)\}]$ is an open set. The map $s$ is open and $s^2 = id$, hence $s$ is continuous (a homeomorphism).
	\end{proof}

	\begin{theorem}\textup{\cite{KozLeid-2}}\label{theorem-wap-linear}
		Let $(Y,<)$ be an ultrahomogeneous linearly ordered set. Then the WAP compactification of the group $(Aut(Y,<), \tau_\partial)$ is topologically isomorphic to the semigroup of all partial automorphisms of $(Y,<)$ in the Fell topology: $wAut(Y,<) = (pAut(Y,<), \tau_F)$.
	\end{theorem}

	\begin{lemma}\textup{\cite[Lemma~2.35]{KozLeid-2}}\label{lemma-closure-of-group-is-sscomp}
		Let $(S,\cdot)$ be a compact semitopological semigroup, and $(G,\cdot)$ a subsemigroup in $S$. If $G$ is isomorphic to some topological group $H$, then $\ovl{G}$ is a semitopological semigroup compactification of the group $H$.
	\end{lemma}

	\begin{theorem}\label{theorem-wG-description}
		$wG = b_XG$.
	\end{theorem}
	
	\begin{proof}
		By Proposition~\ref{prop-b_XG-description}, $b_XG$ is a semitopological semigroup compactification of $G$. Consequently, $wG \geq b_XG$, i.e., there exists a compactification map $\pi: wG \to b_XG$ that is a semigroup homomorphism. It suffices to show that $\pi$ is injective. The inequality also implies that $WAP(G) \supset \A$, where $WAP(G)$ is the algebra of WAP functions on $G$, and $\A$ is the algebra of all continuous bounded functions on $G$ that extend continuously to $b_XG$.
		
		\smallskip
		
		\textbf{$b_XG$ is a Rees quotient semigroup of $wG$.} We show that every function $\phi \in WAP(G)$ extends continuously to $b_XG \sm \{0\}$. Here, $0$ denotes the zero element of the semigroup $b_XG$ (the empty partial automorphism of $(X,R)$).
		
		Fix $x \in X$. Then the action $St_x \acts (Y,<) = (X \sm \{x\},<_x)$ is ultratransitive. The stabilizer $St_x = \{g \in G \;\vert\; gx = x\}$ is a topological group. For a partial automorphism $p \in b_XG$, by the definition of the Fell topology $\tau_F$, the closure of the stabilizer $St_x$ in $b_XG$ has the form $\ovl{St_x}^b = \{ p \in b_XG \;\vert\; x \in D(p),\; px = x\}$ and is a semitopological semigroup by Lemma~\ref{lemma-closure-of-group-is-sscomp}. The map $\eta: (pAut(Y,<),\tau_F) \to \ovl{St_x}^b : p \mapsto p \cup \{(x,x)\}$ is a topological isomorphism, and $\eta|_{Aut(Y,<)} : (Aut(Y,<),\tau_\partial) \to St_x$ is a topological group isomorphism. By Lemma~\ref{lemma-closure-of-group-is-sscomp} and Theorem~\ref{theorem-wap-linear}, we have $\ovl{St_x}^b = wSt_x$.
		
		The orbit $Orb_G(\phi) = \{ \phi_g \;\vert\; g \in G\}$ of the function $\phi$, where $\phi_g(f) = \phi(gf)$, is weakly relatively compact in $C_b(G)$. Then the orbit $Orb_{St_x}(\phi)$ is weakly relatively compact, so $\phi|_{St_x} \in WAP(St_x)$. Consequently, $\phi|_{St_x}$ extends continuously to a function $\ovl{\phi|_{St_x}}$ on $wSt_x = \ovl{St_x}^b$. For each $g \in G$, we have $\phi_g \in WAP(G)$, so $\phi_g|_{St_x}$ extends continuously to $\ovl{\phi_g|_{St_x}}$ on $\ovl{St_x}^b$.
		
		The action $G \acts X$ is transitive, so $G = \bigsqcup\limits_{y \in X} g_ySt_x$, where $g_y x = y$ (the cosets $g_ySt_x$ and $g_zSt_x$ intersect if and only if $g_yx = g_zx$, i.e. $y = z$).

		Define an extension $\phi: G \to \Cb$ to $\psi_x: \bigcup\limits_{y \in X} g_y\ovl{St_x}^b \to \Cb$. Since $\phi|_{gSt_x} = \phi_g|_{St_x}$, we set $\psi_x|_{g_ySt_x} = \phi_{g_y}|_{St_x}$ (i.e., $\psi_x(g_yf) := \phi_{g_y}(f)$) and $\psi_x|_{g_y\ovl{St_x}^b} = \ovl{\phi_{g_y}|_{St_x}}$ (i.e., $\psi_x(g_yp) := \phi_{g_y}(p)$). Then $\psi_x|_G = \phi$. Since $g_y\ovl{St_x}^b = \{p \in b_XG \;\vert\; x \in D(p), px = y\} = b_XG \cap \{(x,y)\}^-$ is clopen in $b_XG$, and for $y \neq z$, $g_y\ovl{St_x}^b \cap g_z\ovl{St_x}^b = \em$, the map $\psi_x$ is well-defined and continuous on $D_x := \bigcup\limits_{y \in X} g_y\ovl{St_x}^b = \{p \in b_XG \;\vert\; x \in D(p)\} = \left( \bigcup\limits_{y \in X} \{(x,y)\}^- \right) \cap b_XG$ (open in $b_XG$).
		
		By similar reasoning, for each $x \in X$, we obtain a family $\{ \psi_x: D_x \to \Cb \;\vert\; x \in X\}$ of continuous extensions of $\phi$. If $x \neq y$, then the functions $\psi_x$ and $\psi_y$ are compatible: $D_x$ and $D_y$ are open, so $G \cap (D_x \cap D_y)$ is dense in $D_x \cap D_y$; $\psi_x|_G = \phi = \psi_y|_G$. Therefore, the combination of maps $\psi := \Nabla\limits_{x \in X} \psi_x : \bigcup\limits_{x \in X} D_x \to \Cb$ is a continuous extension of $\phi$ to $D := \bigcup\limits_{x \in X} D_x = \{p \in b_XG \;\vert\; D(p) \neq \em\}$. We have $b_XG \sm D = \{0\}$. Every $\phi \in WAP(G)$ extends uniquely to $\psi: b_XG \sm \{0\} \to \Cb$, equivalently, for the compactification map $\pi: wG \to b_XG$, we have $|\pi^{-1}(p)| = 1 \; \forall p \in b_XG \sm \{0\}$. That is, $b_XG$ is a Rees quotient semigroup of $wG$.

		Henceforth, we identify the sets $b_XG \sm \{0\}$ and $wG \sm \pi^{-1}(0). \hfill (\diamond)$
		
		\smallskip
		
		\textbf{$wG$ is a semigroup with zero $z$.}
		
		For each $x \in X$, denote by $H_x = \{ p \in b_XG \;\vert\; x \notin D(p), p \cup \{(x,x)\} \in G\}$ a topological group isomorphic to $St_x$ and $Aut(X \sm \{x\}, <_x)$. The closure $\ovl{H_x}^b$ is a semitopological semigroup isomorphic to $pAut(X \sm \{x\}, <_x)$. By Theorem~\ref{theorem-wap-linear}, $\ovl{H_x}^b = wH_x$. Note that $0 \in \ovl{H_x}^b \sm H_x$.
		
		Viewing $H_x$ as a subset of $wG$ by virtue of ($\diamond$), we obtain a semitopological semigroup compactification $S_x := \ovl{H_x}^w$ of the group $H_x$. Then $\pi|_{S_x}: S_x \to \ovl{H_x}^b = wH_x$ is a map of semitopological semigroup compactifications of $H_x$. Consequently, $\pi|_{S_x}$ is an isomorphism of the semigroups $S_x$ and $wH_x$. In particular, the semigroup $S_x$ contains a zero element $z_x$; $\pi(z_x) = 0$.

		Let $x \neq y$. Consider points $a, b \in X$ distinct from $x$ and $y$. Denote $d_a = \{(a,a)\}$ and $d_b = \{(b,b)\}$ --- partial automorphisms. Then $d_a, d_b \in S_x \cap S_y$. Since $d_a d_b = 0$ in $b_XG$, it follows that $d_a d_b = z_x$ in $S_x$, and $d_a d_b = z_y$ in $S_y$. Hence, $z_x = z_y$. Therefore, we can define $z := z_x = z_y$ for all $x, y \in X$, with $\pi(z) = 0$.
		
		Fix $g \in G$. Since $\pi|_{S_x}$ is an isomorphism, we have $\pi^{-1}(0) \cap S_x = \{z\}$. Multiplying this equality by $g$, we obtain $\pi^{-1}(0) \cap gS_x = \{gz\}$. Fix $a \in X \sm \{x, gx\}$. There exists $y \in (a,x) \cap (a,gx)$. Then $d_a, gd_x = \{(x,gx)\} \in S_y$. We have $z = (gd_x) d_a = g(d_x d_a) = gz$. Therefore, $\forall g \in G$, $gz = z$. Similarly, $zg = z$. By the density of $G$ in $wG$ and the separate continuity of multiplication in $wG$, the point $z$ is the zero element in $wG$.
		
		\smallskip
		
		\textbf{We have $|\pi^{-1}(0)| = 1$.}
		
		Since $b_rG \geq wG$, there exists a compactification map $F: b_rG \to wG$. By Theorem~\ref{theorem-roelcke-equals-beta}, we identify $b_rG$ with $b_{\beta_GX} G$. For $T := \pi \circ F: b_rG \to b_XG$, we have $T(A) = A \cap X^2$. Then $T^{-1}(0) = \{A \in b_rG \;\vert\; A \cap X^2 = \em\}$.
		
		In what follows, $a, b, c, d, \dots$ denote points of the set $\beta_GX \sm X$.
		
		Define $\Cc_1 = \{ C_a^b \;\vert\; a, b \in \beta_GX \sm X\}$, where $C_a^b := \{a\} \times \beta_GX \cup \beta_GX \times \{b\}$. For each $n \in \Nb$ ($n \geq 2$), define
		\[
		\Cc_n = \Cc_{n-1} \cup \{ C_{\ovl{a}}^{\ovl{b}} \;\vert\; \ovl{a} = (a_1, \dots, a_n),\ \ovl{b} = (b_1, \dots, b_n) \text{ are cycles in } \beta_GX \sm X \} \subset T^{-1}(0),
		\]
		where
		\begin{align*}
			C_{\ovl{a}}^{\ovl{b}} = [a_1,a_2] \times \{b_1\} \cup [a_2,a_3] \times \{b_2\} \cup \dots \cup [a_{n-1},a_n] \times \{b_{n-1}\} \cup [a_n,a_1] \times \{b_n\} \cup \\
			\cup \{a_1\} \times [b_n, b_1] \cup \{a_2\} \times [b_1, b_2] \cup \dots \cup \{a_n\} \times [b_{n-1}, b_n].
		\end{align*}

		We show that $F(\Cc_n) = \{z\}$ for all $n \geq 1$. We proceed by induction.
		
		\textbf{Base case: $n = 1$.} The set $\Cc_1$ is a closed left- and right-invariant subset of $b_rG$ (since $gC_a^b = C_a^{g b}$ and $C_a^b g = C_{g^{-1}a}^b$). Consequently, $F(\Cc_1)$ is an ideal in $wG$, so $z \in F(\Cc_1)$.
		
		For $d \neq a$, define $C_a^b \ast C_c^d := C_c^b$. For any net $G \ni g_\alpha \to C_a^b$, we have $\lim\limits_{\alpha} (g_\alpha d) = b$. Then $C_a^b \ast C_c^d = C_c^b = \lim\limits_{\alpha} (C_c^{g_\alpha d}) = \lim\limits_{\alpha} (g_\alpha C_c^d)$. We obtain the following chain of equalities:
		\begin{align*}
			F(C_c^b) &= F(C_a^b \ast C_c^d) = F\left(\lim\limits_{\alpha} (g_\alpha C_c^d)\right) = \lim\limits_{\alpha} F(g_\alpha C_c^d) = \lim\limits_{\alpha} \left(F(g_\alpha) F(C_c^d)\right) \\
			&= F\left(\lim\limits_{\alpha} g_\alpha\right) F(C_c^d) = F(C_a^b) F(C_c^d),
		\end{align*}
		where we used the continuity and equivariance of $F$, as well as the continuity of right multiplication by the element $F(C_c^d)$ (since $wG$ is a semitopological semigroup).
		
		Since $z \in F(\Cc_1)$, there exist $d, e$ such that $F(C_d^e) = z$. If $d = e$, then $z = F(C_e^e)$, and for all $a$ and some $b \neq e$, we have $z = z F(C_a^b) = F(C_e^e) F(C_a^b) = F(C_a^e)$. If $d \neq e$, then $z = z F(C_a^e) = F(C_d^e) F(C_a^e) = F(C_a^e)$ for all $a$. In each case, for any $a, b$ and some $c \neq e$, we have $z = F(C_c^b) z = F(C_c^b) F(C_a^e) = F(C_a^b).$ Therefore, $F(\Cc_1) = \{z\}$.

		\textbf{Induction step.} Suppose $F(\Cc_{n-1}) = \{z\}$. For cycles $\ovl{a}, \ovl{b}, \ovl{c}, \ovl{d}$ such that $d_i \in (a_i, a_{i+1})$ for all $i = 1, \dots, n$ (indices are taken modulo $n$), define
		\[
		C_{\ovl{a}}^{\ovl{b}} \ast C_{\ovl{c}}^{\ovl{d}} := C_{\ovl{c}}^{\ovl{b}}.
		\]
		Similarly to the case $n = 1$, for any net $G \ni g_\alpha \to C_{\ovl{a}}^{\ovl{b}}$, we have $F(C_{\ovl{c}}^{\ovl{b}}) = F(C_{\ovl{a}}^{\ovl{b}} \ast C_{\ovl{c}}^{\ovl{d}}) = F\left(\lim\limits_{\alpha} (g_\alpha C_{\ovl{c}}^{\ovl{d}})\right) = F(C_{\ovl{a}}^{\ovl{b}}) F(C_{\ovl{c}}^{\ovl{d}})$.
		
		Fix $a_2 \in X^+$. Replacing the element $d_2$ of the cycle $\ovl{d}$ in the equality $F(C_{\ovl{c}}^{\ovl{b}}) = F(C_{\ovl{a}}^{\ovl{b}}) F(C_{\ovl{c}}^{\ovl{d}})$ with a net $d_\alpha$, where $(a_2, a_3) \ni d_\alpha \to a_2$, we obtain
		\[
		F(C_{\ovl{c}}^{\ovl{b}}) = F(C_{\ovl{a}}^{\ovl{b}}) F(C_{\ovl{c}}^{d_1, d_\alpha, d_3, \dots, d_n}) \to F(C_{\ovl{a}}^{\ovl{b}}) F(C_{\ovl{c}}^{d_1, a_2, d_3, \dots, d_n}).
		\]
		
		Replacing the element $a_3$ of the cycle $\ovl{a}$ in the equality $F(C_{\ovl{c}}^{\ovl{b}}) = F(C_{\ovl{a}}^{\ovl{b}}) F(C_{\ovl{c}}^{d_1, a_2, d_3, \dots, d_n})$	with a net $a_\beta$, where $(a_2, a_4) \ni a_\beta \to a_2$, we get
		\[
		F(C_{\ovl{c}}^{\ovl{b}}) = F(C_{a_1, a_2, a_\beta, a_4, \dots, a_n}^{\ovl{b}}) F(C_{\ovl{c}}^{d_1, a_2, d_3, \dots, d_n}) \to F(C_{a_1, a_2, a_4, \dots, a_n}^{b_1, b_3, b_4, \dots, b_n}) F(C_{\ovl{c}}^{d_1, a_2, d_3, \dots, d_n}) = z,
		\]
		since $C_{a_1, a_2, a_4, \dots, a_n}^{b_1, b_3, b_4, \dots, b_n} \in \Cc_{n-1}$ and $F(\Cc_{n-1}) = \{z\}$ by the induction hypothesis. Then $F(C_{\ovl{c}}^{\ovl{b}}) = z$. By the arbitrariness of the cycles $\ovl{c}$ and $\ovl{b}$, we have $F(\Cc_{n}) = \{z\}$.
		
		Since the set $\bigcup\limits_{n \in \Nb} \Cc_n$ is dense in $T^{-1}(0)$, it follows that $\{z\} = F(\bigcup\limits_{n \in \Nb} \Cc_n)$ is dense in $F(T^{-1}(0)) = \pi^{-1}(0)$. Therefore, $\pi^{-1}(0) = \{z\}$.
	\end{proof}

	\begin{remark}
		The WAP compactification $wG$ has a zero element and a continuous involution that extends the inversion on $G$.
	\end{remark}
	
	\begin{remark}
		Arguing similarly to \cite{KozLeid-2}, one can show that there exists a homomorphic embedding $i: G \to U(l_2)$ into the unitary group of a Hilbert space, such that $\ovl{i(G)}^{l_2} = wG$.
	\end{remark}

	\section{Corollaries}
	
	\subsection{Equivalent descriptions of \texorpdfstring{$wAut(X,R)$}{wAut(X,R)}}
	
	\begin{propos}\cite[Theorem 6]{GSorin26}\label{prop-e-alphaXG-description}
		The Ellis compactification $e_{\alpha X}G$ is the set of all $f \in \alpha X^{\alpha X}$ satisfying the following properties:
		\begin{enumerate}
			\item $f(\infty) = \infty$;
			\item $\forall a \neq b \in X \; (fa = fb \implies fa = \infty)$;
			\item $\forall a,b,c \in X \; (fa, fb, fc \in X \implies ([fa, fb, fc] \iff [a,b,c]))$.
		\end{enumerate}
		$e_{\alpha X}G$ is a semitopological semigroup.
	\end{propos}
	
	We have the following equivalent descriptions of $wG$.
	
	\begin{coro}\label{coro-small-comp-equality}
		$wG = b_{\alpha X}G = e_{\alpha X}G$.
	\end{coro}
	
	\begin{proof}
		Since the discrete space $X$ is locally compact and locally connected, we have $b_XG = b_{\alpha X}G$ \cite[Theorem 3.33]{KozLeid-2}. By Theorem~\ref{theorem-wG-description}, $wG = b_{\alpha X}G$.
		
		By Proposition~\ref{prop-e-alphaXG-description}, $e_{\alpha X} G$ is a semitopological semigroup compactification of $G$, so $wG \geq e_{\alpha X} G$. The reverse inequality follows from the fact that the map $\Psi: e_{\alpha X} G \to b_XG: f \mapsto \Gamma(f) \cap X^2$ is a compactification map.
	\end{proof}
	
	An explicit description of the compactification $b_{\alpha X} G$ follows from the equality $b_XG = \{A \cap X^2 \;\vert\; A \in b_{\alpha X} G\}$ and the condition $\forall x \in \alpha X \;\exists y \in A[x] \;\exists z \in sA[x]$.
	
	\begin{propos}\label{prop-b_alphaXG-description}
		$b_{\alpha X}G := \ovl{(i_{\alpha X}(G))}$ is the set of all $A \in 2^{\alpha X \times \alpha X}$ satisfying the following properties:
		\begin{enumerate}
			\item[(i)] $\forall x \in X \;\exists! y \in \alpha X: y \in A[x]$; $\forall y \in X \;\exists! x \in \alpha X: x \in sA[y]$;
			\item[(ii)] $(\infty, \infty) \in A$; $pr_1(A \sm X^2) = \alpha X \sm D(A)$; $pr_2(A \sm X^2) = \alpha X \sm I(A)$;
			\item[(iii)] $\forall x_1,x_2,x_3 \in D(A) \; ([x_1,x_2,x_3] \implies [A[x_1],A[x_2],A[x_3]])$.
		\end{enumerate}
	\end{propos}

	\subsection{Semitopological semigroup compactifications of \texorpdfstring{$Aut(X,R)$}{Aut(X,R)}}
	
	In \cite{Maltsev}, a description of all (two-sided) ideals of the symmetric inverse semigroup $I_X$ is given. Recall that for a partial bijection $p \in I_X$, the image of $p$ is defined as $I(p) = \{ y \in X \;\vert\; \exists x \in X: (x,y) \in p \}$. For an arbitrary cardinal $\xi \leq |X|$, the set
	\[
	I^0_\xi = \{p \in I_X \;\vert\; |I(p)| < \xi\}
	\]
	is an ideal.
	
	Since $b_XG$ is a subsemigroup of $I_X$ consisting of all partial automorphisms, the sets
	\[
	I_\xi = I^0_\xi \cap b_XG = \{p \in b_XG \;\vert\; |I(p)| < \xi\} \subset b_XG \sm G,
	\]
	where $\xi \leq |X|$, are proper ideals in $b_XG$.
	
	\begin{propos}\label{prop-ideals-in-b_XG}
		The ideal $I_\xi$ is closed in $b_XG$ if and only if the cardinal $\xi$ is finite.
	\end{propos}
	
	\begin{proof}
		Let $\xi = n$. Fix $p \notin I_n$. Then $|D(p)| = |I(p)| \geq n$, so $p$ contains $n$ distinct points $(x_i, y_i)$ and $p \in W = \bigcap\limits_{i=1}^n \{(x_i, y_i)\}^-$. Since $W \cap I_n = \em$, the set $I_n$ is closed.
		
		Conversely, let $\xi \geq \omega$. It suffices to show that $I_\xi$ is dense in $b_XG$ (since the ideal is proper). A basis element of the topology on $b_XG$ has the form $W = \bigcap\limits_{i=1}^n \{(x_i,y_i)\}^- \cap \bigcap\limits_{j=1}^m \{(a_j,b_j)\}^+$. Since $W \neq \em$, there exists $g \in G \cap W$. Then the partial automorphism $p = g|_{\{x_1,\dots, x_n\}} \in W \cap I_\xi$.
	\end{proof}
	
	The group $G$ has a countable family of semitopological semigroup compactifications with a continuous involution.
	
	\begin{coro}\label{coro-rees-quotients}
		For each $n \in \Nb$, the Rees quotient semigroup $b_XG/I_n$ is a semitopological semigroup compactification of $G$. If $n < m$, then $b_XG/I_n > b_XG/I_m$.
		
		If $X$ is countable, then the compactifications $b_XG$ and $b_XG/I_n$, $n \in \Nb$, are homeomorphic to the Cantor set.
	\end{coro}
	
	\begin{proof}
		The first statement follows from the inclusion $I_n \subset b_XG \sm G$, as well as Propositions~\ref{prop-b_XG-description} and~\ref{prop-ideals-in-b_XG}.
		
		Obviously, if $n < m$, then the quotient map $q: b_XG/I_n \to b_XG/I_m$ defines the inequality $b_XG/I_n \geq b_XG/I_m$ (since $I_n \subset I_m$) for semitopological semigroup compactifications of $G$. We now show that $b_XG/I_n \neq b_XG/I_m$. Suppose the contrary; then $q$ is both a homeomorphism and a semigroup isomorphism. Let $A \subset X$ with $|A| = n$. Then the partial automorphism $id_A \in I_m \sm I_n$. Denote by $f_n$ (respectively $f_m$) the compactification map $b_XG \to b_XG/I_n$ (respectively $b_XG \to b_XG/I_m$), which is a homomorphism; then $qf_n = f_m$. Since $f_n(id_A) \neq 0_n$ and $f_m(id_A) = 0_m$, we have $q^{-1}(0_m) \neq 0_n$. This is a contradiction.
		
		The last statement follows from Corollary~\ref{coro-small-comp-equality}: $b_XG = e_{\alpha X}G \subset \alpha X^{\alpha X}$. If $X$ is countable, then $\alpha X$ is metrizable, so $\alpha X^{\alpha X}$ and $b_XG$ are metrizable. The compactifications $b_XG/I_n$, $n \in \Nb$, are metrizable as perfect images of the metrizable compact space $b_XG$. The spaces $b_XG$ and $b_XG/I_n$, $n \in \Nb$, are zero-dimensional, metrizable compacta without isolated points; hence, they are homeomorphic to the Cantor set.
	\end{proof}

	\subsection{The WAP compactification of the group \texorpdfstring{$(Aut(X,R), \tau_p)$}{(Aut(X,R), tau\_p)} is trivial}
	
	Equip the c.o. set $(X,R)$ with the interval topology $\tau_R$, whose basis consists of all intervals $(a,b)$, where $a,b \in X$. The group $H = (Aut(X,R), \tau_p)$, endowed with the topology of pointwise convergence for its action on $(X,\tau_R)$, is a topological group, and the action $H \acts (X,\tau_R)$ is continuous \cite[Theorem 2]{GSorin25}.
	
	\begin{theorem}\label{theorem-wap-trivial}
		The WAP compactification $wH$ is trivial.
	\end{theorem}
	
	\begin{proof}
		Let $S$ be a compact semitopological semigroup, and let $f: H \to S$ be a continuous semigroup homomorphism with dense image. It suffices to show that $|S| = 1$.
		
		Since the action $H \acts (X,R)$ is ultratransitive, for each point $x \in X$, the action $St_x \acts (X \sm \{x\}, <_x)$ is ultratransitive. The topology of pointwise convergence for the action $St_x \acts (X \sm \{x\}, \tau_{<_x})$ coincides with the original topology on $St_x$ (here $\tau_{<_x}$ denotes the interval topology induced by the linear order $<_x$). By \cite[Theorem 5.10]{KozLeid-2}, the WAP compactification of the group $St_x$ is trivial. Consequently, $f(St_x) = \{u_x\}$ for all $x \in X$.
		
		Since $St_x \cap St_y \neq \em$ for all $x,y \in X$, we have $u := u_x = u_y$. For the identity element $e \in H$, we have $f(e) = u$. Therefore, $u$ is an idempotent.
		
		We have $f\left(\bigcup\limits_{x \in X} St_x\right) = \{u\}$. Let $g \in H \sm \bigcup\limits_{x \in X} St_x$. Then $g$ has no fixed points; in particular, for some $x_0 \in X$, we have $gx_0 = y_0 \neq x_0$. Fix $z_0 \in X \sm \{x_0, y_0\}$. By ultratransitivity, there exists $h_2 \in St_{z_0}$ such that $h_2 x_0 = y_0$. Set $h_1 = gh_2^{-1}$. Then $h_1 \in St_{y_0}$. Hence, $g = h_1 h_2$, and $f(g) = f(h_1) f(h_2) = u^2 = u$. Thus, $f(H) = \{u\}$, and the density of $f(H)$ in $S$ implies $|S| = 1$.
	\end{proof}
	
	\begin{remark}
		The triviality of the WAP compactification of the automorphism group $(Aut(\Tb, R), \tau_p)$ of the circle is proved in \cite[Remark 7.5.2]{GlasMegTame}.
	\end{remark}
	
	\newpage
	
\end{fulltext}

\end{document}